\documentclass[12pt]{article}

\usepackage{enumerate}
\usepackage{amsmath}
\usepackage{amssymb,latexsym}
\usepackage{amsthm}
\usepackage{color}

\usepackage{graphicx}

\usepackage{hyperref}

\usepackage{amscd}
\usepackage[normalem]{ulem}

\DeclareMathOperator{\C}{\mathcal{C}}

\DeclareMathOperator{\NNN}{N}

\newtheorem{theorem}{Theorem}[section]

\newtheorem{lemma}[theorem]{Lemma}
\newtheorem{corollary}[theorem]{Corollary}

\newtheorem{proposition}[theorem]{Proposition}

\newtheorem{example}[theorem]{Example}
\newtheorem{conjecture}[theorem]{Conjecture}

\newtheorem{remark}[theorem]{Remark}

\newcommand{\Fq}{\mathbb F_q}
\newcommand{\Fqn}{\mathbb F_{q^n}}

\newcommand{\cG}{{\mathcal G}}

\newcommand{\p}{{\mathrm p}}
\newcommand{\N}{{\mathrm N}}

\newcommand{\cH}{{\mathcal H}}
\newcommand{\cC}{{\mathcal C}}

\newcommand{\cL}{{\mathcal L}}

\newcommand{\F}{{\mathbb F}}

\newcommand{\Tr}{\hbox{{\rm Tr}}}

\newcommand{\la}{\langle}
\newcommand{\ra}{\rangle}

\newcommand{\PG}{\mathrm{PG}}

\title{Vertex properties of maximum scattered linear sets of $\PG(1,q^n)$}
\author{Corrado Zanella and Ferdinando Zullo \thanks{
The
research  was supported by
 the Italian National
Group for Algebraic and Geometric Structures and their Applications (GNSAGA
- INdAM). }}
\date{}

\begin{document}
\maketitle

\begin{abstract}
In this paper we investigate the geometric properties of the configuration consisting
of a subspace $\Gamma$ and a canonical subgeometry $\Sigma$ in $\PG(n-1,q^n)$,
with $\Gamma\cap\Sigma=\emptyset$.
The idea motivating is that such properties are reflected in the algebraic structure
of the linear set which is projection of $\Sigma$ from the vertex $\Gamma$.
In particular
we deal with the maximum scattered linear sets of the line $\PG(1,q^n)$ found by Lunardon and Polverino in \cite{LP2001} and recently generalized by Sheekey in \cite{Sh}.
Our aim is to characterize this family by means of the properties of the vertex of the projection as done by Csajb\'ok and the first author of this paper for linear sets of pseudoregulus type.
With reference to such properties, we construct new examples of scattered linear sets in $\PG(1,q^6)$, yielding also to new examples of MRD-codes in $\F_q^{6\times 6}$ with left idealiser isomorphic to $\F_{q^6}$.
\end{abstract}



\section{Introduction}

Let $\Lambda=\PG(W,\F_{q^n})=\PG(1,q^n)$, where $W$ is a vector space of dimension $2$ over $\F_{q^n}$.
A point set $L$ of $\Lambda$ is said to be an \emph{$\F_q$-linear set} of $\Lambda$ of rank
$\rho$ if it is
defined by the non-zero vectors of a $\rho$-dimensional $\F_q$-vector subspace $U$ of $W$, i.e.
\[L=L_U=\{\la {\bf u} \ra_{\mathbb{F}_{q^n}} \colon {\bf u}\in U\setminus \{{\bf 0} \}\}.\]
Two linear sets $L_U$ and $L_W$ of $\PG(1,q^n)$ are said to be \emph{$\mathrm{P\Gamma L}$-equivalent} if there is an element $\phi$ in $\mathrm{P\Gamma L}(2,q^n)$ such that $L_U^{\phi} = L_W$.
It may happen that two $\F_q$--linear sets $L_U$ and $L_W$ of $\PG(1,q^n)$ are $\mathrm{P\Gamma L}$-equivalent even if the $\F_q$-vector subspaces $U$ and $W$ are not in the same orbit of $\Gamma \mathrm{L}(2,q^n)$ (see \cite{CSZ2015} and \cite{CMP} for further details).

Lunardon and Polverino in \cite{LuPo2004} (see also \cite{LuPoPo2002}) show that every linear
set is a projection of a canonical subgeometry, where a \emph{canonical subgeometry} in $\PG(m-1,q^n)$ is a linear set $L$ of rank $m$ such that $\la L \ra=\PG(m-1,q^n)$ (\footnote{Angle brackets without the indication of a field will denote the projective span of a set of points in a projective space.}).
In particular, by \cite[Theorems 1 and 2]{LuPo2004} (adapted to the projective line case), for each $\F_q$-linear set $L_U$ of the projective line $\Lambda=\PG(1,q^n)$ of rank $n$ there exist a canonical subgeometry $\Sigma\cong\PG(n-1,q)$ of $\Sigma^*=\PG(n-1,q^n)$,
and an $(n-3)$-subspace $\Gamma$ of $\Sigma^*$ disjoint from $\Sigma$ and from $\Lambda$ such that
\[ L_U=\p_{\Gamma,\Lambda}(\Sigma)=\{\langle \Gamma,P\rangle \cap \Lambda \colon P \in \Sigma\}. \]
We call $\Gamma$ and $\Lambda$ the \emph{vertex} (or \emph{center}) and the \emph{axis} of the projection, respectively.

In this paper we focus on {\it maximum scattered} $\F_q$-linear sets of $\PG(1,q^n)$, that is,
$\F_q$-linear sets of rank $n$ in $\PG(1,q^n)$ of size $(q^n-1)/(q-1)$.
In this case, we also say that the related $\F_q$-subspace is \emph{maximum scattered}.
Recall that the \emph{weight of a point} $P=\langle \mathbf{u} \rangle_{\F_{q^n}}$ is $w_{L_U}(P)=\dim_{\F_q}(U\cap\langle \mathbf{u} \rangle_{\F_{q^n}})$. A linear set $L_U$ is scattered if and only if each of its points has weight one.

If $\la (0,1) \ra_{\F_{q^n}}$ is not contained in the linear set $L_U$ of rank $n$ of $\PG(1,q^n)$ (which we can always assume after a suitable projectivity), then $U=U_f:=\{(x,f(x))\colon x\in \F_{q^n}\}$ for some linearized polynomial (or \textit{$q$-polynomial}) $f(x)=\sum_{i=0}^{n-1}a_ix^{q^i}\in \F_{q^n}[x]$. In this case we will denote the associated linear set by $L_f$.

The first example of maximum scattered $\F_q$-linear set, found by Blokhuis and Lavrauw in  \cite{BL2000}, is known as linear sets of \emph{pseudoregulus type} and
 can be defined (see \cite[Section 4]{LuMaPoTr2014}) as any linear set
 $\mathrm{P\Gamma L}$-equivalent to
\[ L^1=\{\la (x,x^q) \ra_{\mathbb{F}_{q^n}} \colon x \in \F_{q^n}^*\}. \]

A characterization of the linear sets of pseudoregulus type has been given by Csajb\'ok and Zanella in \cite{CsZ20162} as particular projections of a canonical subgeometry (see Theorem \ref{chPseudo}).

\begin{theorem}\cite[Theorem 2.3]{CsZ20162}\label{chPseudo}
Let $\Sigma$ be a canonical subgeometry of $\PG(n-1,q^n)$, $q>2$, $n \geq 3$.
Assume that $\Gamma$ and $\Lambda$ are an $(n-3)$-subspace and a line of $\PG(n-1,q^n)$, respectively, such that $\Sigma \cap \Gamma=\Lambda \cap \Gamma= \emptyset$. Then the following assertions are equivalent:
\begin{enumerate}
  \item The set $\p_{\Gamma, \Lambda}(\Sigma)$ is a scattered $\F_q$-linear set of pseudoregulus type;
  \item A generator $\hat{\sigma}$ exists of the subgroup of $\mathrm{P}\Gamma\mathrm{L}(n,q^n)$ fixing $\Sigma$ pointwise, such that $\dim(\Gamma\cap\Gamma^{\hat{\sigma}})=n-4$; furthermore $\Gamma$ is not contained in the span of any hyperplane of $\Sigma$;
  \item There exists a point $P_\Gamma$ and a generator $\hat{\sigma}$ of the subgroup of $\mathrm{P}\Gamma\mathrm{L}(n,q^n)$ fixing $\Sigma$ pointwise, such that $\la P_\Gamma,P_\Gamma^{\hat{\sigma}},\ldots, P_\Gamma^{\hat{\sigma}^{n-1}} \ra=\PG(n-1,q^n)$, and
      \[ \Gamma=\la P_\Gamma,P_\Gamma^{\hat{\sigma}},\ldots, P_\Gamma^{\hat{\sigma}^{n-3}} \ra. \]
\end{enumerate}
\end{theorem}

Few other families of maximum scattered linear sets of $\PG(1,q^n)$ are known, see \cite{CMPZ,CsMZ2018}.
We will deal with the only remaining family of maximum scattered linear sets existing for each value of $n$.
Such a family has been introduced by Lunardon and Polverino in \cite{LP2001} for $s=1$ and generalized by Sheekey in \cite{Sh} and is defined as follows
\begin{equation}\label{LPform}
L_{s,\delta}^n=\{\la(x,\delta x^{q^s} + x^{q^{n-s}})\ra_{\F_{q^n}}\colon x\in \F_{q^n}^*\},
\end{equation}
with $n\geq 4$, $\N_{q^n/q}(\delta)\notin \{0,1\}$
(\footnote{This condition  implies $q\neq 2$.}) and $(s,n)=1$.
More generally, we will call each linear set equivalent to a maximum scattered linear set of the form \eqref{LPform}, with $\delta \neq 0$, of \emph{Lunardon-Polverino type} (or shortly \emph{LP-type}).
For some values of $s$, $\delta$ and $n$, $\N_{q^n/q}(\delta)\notin \{0,1\}$ is a necessary condition for $L_{s,\delta}^n$ to be scattered, see Section \ref{Geo}.
Up to our knowledge, no scattered $L_{s,\delta}^n$ is known satisfying $\N(\delta)=1$.

Our aim is to prove characterizations of maximum scattered linear sets of LP-type in the spirit of the characterization of the linear sets of pseudoregulus type, cf.\ Theorem \ref{chPseudo}.
As a consequence, we will construct new examples of maximum scattered linear sets in $\PG(1,q^6)$.
As showed in \cite[Sect.\ 5]{Sh},
this also yields to new examples of MRD-codes in $\F_q^{6\times 6}$ with left idealiser isomorphic to $\F_{q^6}$ \cite[Proposition 6.1]{CMPZ}
 (see also \cite{CSMPZ2016,CsMPZ2019,ShVdV}), see last section for more details on the connections with MRD-codes.

We will work in the following framework.
Let $x_0,\ldots,x_{n-1}$ be the homogeneous coordinates of $\PG(n-1,q^n)$ and let
\[ \Sigma=\{\la (x,x^{q},\ldots,x^{q^{n-1}}) \ra_{\F_{q^n}} \colon x \in\F_{q^n} \} \]
be a fixed canonical subgeometry of $\PG(n-1,q^n)$.
The collineation $\hat{\sigma}$ of $\PG(n-1,q^n)$ defined by
$\la(x_0,\ldots,x_{n-1})\ra_{\F_{q^n}}^{\hat{\sigma}}=\la(x_{n-1}^{q},x_0^{q},\ldots,x_{n-2}^{q})\ra_{\F_{q^n}}$ fixes precisely the points of $\Sigma$.
Note that if $\sigma$ is a collineation of $\PG(n-1,q^n)$ such that $\mathrm{Fix}(\sigma)=\Sigma$, then $\sigma=\hat{\sigma}^s$, with $(s,n)=1$.

\section{Possible configurations of the vertex of the projection}

Following \cite[Section 3]{GiuZ}, we are able to
describe the structure of the vertex of the projection, under certain assumptions regarding the dimension of the intersections with some of its conjugates w.r.t. a collineation of $\PG(n-1,q^n)$ fixing the chosen subgeometry pointwise.

We start by recalling the following lemma.

\begin{lemma}\cite[Lemma 3]{Lun99}\label{int}
If $S$ is a nonempty projective subspace of dimension $k$ of $\PG(n-1,q^n)$ fixed by $\sigma$, then
$S$ meets $\Sigma$ in an $\F_q$-subspace of dimension $k$.
In particular, $S \cap \Sigma \neq \emptyset$.
\end{lemma}

\smallskip

Since the vertex of the projection is disjoint from $\Sigma$, we have that $\dim(\Gamma \cap \Gamma^\sigma)\leq \dim\Gamma-1$.
We characterize the extremal case, i.e. when $\dim(\Gamma \cap \Gamma^\sigma)= \dim \Gamma-1$.

\begin{theorem}\label{k-1}
Let $\Gamma$ be a subspace of $\PG(n-1,q^n)$ of dimension $k$ and
such that $\Gamma \cap \Sigma=\emptyset$.
If $\dim (\Gamma \cap \Gamma^\sigma)=k-1$, then there exists exactly one point $P$
in $\PG(n-1,q^n)$ such that
\[ \Gamma=\la P,P^\sigma,\ldots,P^{\sigma^k} \ra. \]
Furthermore, $P^{\sigma^{n-1}}\not\in\Gamma$.
\end{theorem}
\begin{proof}
The hypotheses imply $k\ge0$.
For $k=0$, $\Gamma$ is a point $P$.
If $P^{\sigma^{n-1}}\in\Gamma$, then $P\in\Sigma$, a contradiction.
The remaining statements are trivial for this $k$.

Now suppose that the assertion is true for $(k-1)$-dimensional subspaces, and $k\ge1$.
Let denote by $\Omega=\Gamma\cap\Gamma^\sigma$. Clearly, $\langle\Omega,\Omega^\sigma\rangle \subseteq \Gamma^\sigma$ and $\dim \Gamma^\sigma=k$.
By our assumption, $\dim \Omega=k-1$ and also $\dim (\Omega \cap \Omega^\sigma)=k-2$.
Indeed,
\[ \dim(\Omega\cap\Omega^\sigma)=2(k-1)-\dim \la \Omega, \Omega^\sigma\ra\geq 2k-2-k=k-2.  \]
So,
\[ k-2 \leq \dim(\Omega\cap\Omega^\sigma) \leq k-1, \]
and since $\Omega\neq\Omega^\sigma$, otherwise by Lemma \ref{int} we should have $\Gamma \cap \Sigma \neq \emptyset$, we get $\dim(\Omega\cap\Omega^\sigma)=k-2$.
Therefore, there exists a point $P' \in \Omega$ such that
\[ \Omega=\Gamma\cap\Gamma^\sigma=\la P',P'^\sigma,\ldots,P'^{\sigma^{k-1}} \ra. \]
By induction hypothesis, $P'^{\sigma^{n-1}} \notin \Gamma\cap \Gamma^\sigma$.
So,
\[ \Gamma=\la P,P^\sigma,\ldots,P^{\sigma^k} \ra, \]
with $P=P'^{\sigma^{n-1}}$.

Regarding the uniqueness, if $\Gamma=\la Q,Q^\sigma,\ldots,Q^{\sigma^k}\ra$ for some point $Q$,
then $\Omega=\la Q^\sigma,\ldots,Q^{\sigma^k}\ra$.
By induction, this implies $Q^\sigma=P'$ above defined, and $Q=P$.

Finally note that $P^{\sigma^{n-1}}\in\Gamma$ would imply $\Gamma^{\sigma^{n-1}}=\Gamma$
and $\Gamma\cap\Sigma\neq\emptyset$, a contradiction.
\end{proof}

The next result follows for $r=1$ from Theorem \ref{k-1}.

\begin{theorem}\label{k-2}
Let $\Gamma$ be a subspace of $\PG(n-1,q^n)$ of dimension $k\ge0$
such that $\Gamma \cap \Sigma=\emptyset$, and $\dim(\Gamma\cap\Gamma^\sigma)\ge k-2$.
Let $r$ be the least positive integer satisfying the condition
\begin{equation}\label{def_r}
\dim(\Gamma\cap\Gamma^\sigma\cap\Gamma^{\sigma^2}\cap\ldots\cap\Gamma^{\sigma^r})>k-2r.
\end{equation}
Then there is a point $P\in\PG(n-1,q)$ satisfying
\begin{enumerate}[(i)]
\item $P$, $P^\sigma$, $\ldots$, $P^{\sigma^{k-r+1}}$ are independent points contained
in $\Gamma$;
\item $P^{\sigma^{n-1}}\not\in\Gamma$.
\end{enumerate}
If $r<(k+2)/2$, then the point $P$ satisfying conditions (i) and (ii) is unique.
\end{theorem}

We will call the integer $r$ of the above statement the \emph{intersection number of} $\Gamma$ w.r.t.\ $\sigma$ and we will denote it by $\mathrm{intn}_{\sigma}(\Gamma)$.

\begin{proof}
\textit{Preliminary remarks.}
Since $\sigma$ is a collineation and since $\dim (\Gamma \cap \Gamma^\sigma)\geq k-2$, for any positive integer $t$ it holds
\begin{gather*}
\dim(\Gamma\cap\Gamma^\sigma\cap\ldots\cap\Gamma^{\sigma^{t+1}})=\\
\dim\Gamma+\dim(\Gamma^\sigma\cap\ldots\cap\Gamma^{\sigma^{t+1}})
-\dim\left(\la\Gamma\cup(\Gamma^\sigma\cap\ldots\cap\Gamma^{\sigma^{t+1}})\ra\right)\ge\\
\dim\Gamma+\dim(\Gamma\cap\ldots\cap\Gamma^{\sigma^t})
-\dim\left(\la\Gamma\cup\Gamma^\sigma\ra\right)\ge
\dim(\Gamma\cap\ldots\cap\Gamma^{\sigma^t})-2.
\end{gather*}
This implies $\dim(\Gamma\cap\ldots\cap\Gamma^{\sigma^t})=k-2t$ for any $0\le t<r$;
so, taking $t=r-1$, $k-2(r-1)\ge-1$, that is $r\le(k+3)/2$.
Furthermore, if $\dim(\Gamma\cap\ldots\cap\Gamma^{\sigma^t})\ge0$,
then
\begin{equation}\label{stepconeccez}
\dim(\Gamma\cap\Gamma^\sigma\cap\ldots\cap\Gamma^{\sigma^{t+1}})
\le
\dim(\Gamma\cap\ldots\cap\Gamma^{\sigma^t})-1,
\end{equation}
for otherwise $\Sigma\cap\Gamma\cap\ldots\cap\Gamma^{\sigma^t}\neq\emptyset$.
This implies
\begin{equation}\label{pre-unicita}
\dim(\Gamma\cap\Gamma^\sigma\cap\Gamma^{\sigma^2}\cap\ldots\cap\Gamma^{\sigma^r})=k-2r+1\quad
\mbox{ for }r\neq\frac{k+3}2.
\end{equation}
{Note that for $r=\frac{k+3}2$ then $\dim(\Gamma\cap\Gamma^\sigma\cap\Gamma^{\sigma^2}\cap\ldots\cap\Gamma^{\sigma^r})=k-2r+2=-1$.}

\textit{Existence of $P$, by induction on $r$.}
For $r=1$, the assertion follows from Theorem \ref{k-1}.
Assume then that Theorem \ref{k-2} holds (except possibly for the uniqueness part)
for $r-1$, and $r\ge2$.
Let $\Omega=\Gamma\cap\Gamma^\sigma$ and $\dim\Omega=k-2=:k'$.
If $k'=-1$, then the thesis is trivial.

Now suppose $k'\geq0$.
Then it holds
\[
\dim(\Omega\cap\ldots\cap\Omega^{\sigma^t})=k'-2t
\]
for $t<r-1$, whereas
\[
\dim(\Omega\cap\ldots\cap\Omega^{\sigma^{r-1}})>k'-2(r-1)=k-2r.
\]
By induction hypothesis there is a point $P'\in\PG(n-1,q^n)$ satisfying
\begin{enumerate}[(A)]
\item $P',P'^{\sigma},\ldots,P'^{\sigma^{k'-(r-1)+1}}=P'^{\sigma^{k-r}}$ are
independent points;
\item $P',P'^{\sigma},\ldots,P'^{\sigma^{k-r}}\in\Omega$;
\item $P'^{\sigma^{n-1}}\not\in\Omega$.
\end{enumerate}
Let $P=P'^{\sigma^{n-1}}$.
Then (B) implies that
$P$, $P^\sigma$, $\ldots$, $P^{\sigma^{k-r+1}}$ are points contained
in $\Gamma$;
both (C) and (A) imply that they are independent.
$P^{\sigma^{n-1}}\in\Gamma$
would imply
\[
P^{\sigma^{r-2}},P^{\sigma^{r-1}},\ldots,P^{\sigma^{k-r+1}}\in\Gamma\cap\Gamma^\sigma\cap\ldots
\cap\Gamma^{\sigma^{r-1}},
\]
contradicting $\dim(\Gamma\cap\Gamma^\sigma\cap\ldots \cap\Gamma^{\sigma^{r-1}})=k-2r+2$.

\textit{Uniqueness of $P$.}
By the previous considerations it follows that there exists at least one point $P$ such that $P$, $P^\sigma$, $\ldots$, $P^{\sigma^{k-r+1}}$ are independent points contained in $\Gamma$.
From (\ref{pre-unicita}) it follows that
\[ \Lambda:=\Gamma\cap\Gamma^\sigma\cap\Gamma^{\sigma^2}\cap\ldots\cap\Gamma^{\sigma^r}= \la P^{\sigma^r}, \ldots, P^{\sigma^{k-r+1}} \ra, \]
has dimension
$k-2r+1>-1$.
Furthermore, $\dim(\Lambda \cap \Lambda^\sigma)=k-2r$, otherwise
$\Gamma\cap\Sigma\neq\emptyset$.
It follows that $\Lambda$ satisfies the hypotheses of Theorem \ref{k-1} and hence the point $P$ is unique.
\end{proof}

\begin{remark}\label{L(P)}
It is clear that $P$ is as in the previous result, it follows that
\[ \dim L(P)\geq k-r+2, \]
where $L(P)=\la P,P^\sigma,\ldots,P^{\sigma^{n-1}} \ra$.
\end{remark}

\begin{remark}
A similar idea to the intersection number for a vertex of a linear set has been presented in \cite{NPH} (see also \cite{GiuZ,NPH2}), where the authors used sequences of the dimensions of certain intersections as invariants for rank metric codes. See also the last section.
\end{remark}

\section{Characterization of linear sets of LP-type}

\subsection{Sufficient conditions}
We are now ready to state sufficient conditions for a linear set to be of LP-type.
In the following we denote by $\NNN(-)$ the norm over $\Fq$, for short.

\begin{theorem}\label{charact1}
In $\PG(n-1,q^n)$,  $n\ge4$,
let $\Gamma$ be a subspace of dimension $n-3$, $\Lambda$ a line, and
$\Sigma\cong\PG(n-1,q)$ a canonical subgeometry, such that
$\Gamma \cap \Sigma=\emptyset=\Gamma\cap\Lambda$.
Assume $L=\p_{\Gamma,\Lambda}(\Sigma)$ is a scattered linear set of $\Lambda$.
If $\mathrm{intn}_{\sigma}(\Gamma)=2$ for some
generator $\sigma$ of the subgroup of $\mathrm{P}\Gamma\mathrm{L}(n,q^n)$ fixing $\Sigma$ pointwise, then there exists a unique point $P$ such that
\[ \Gamma=\langle P,P^{\sigma},\ldots,P^{\sigma^{n-4}},Q \rangle. \]
Furthermore, if the line $\langle P^{\sigma^{n-1}}, P^{\sigma^{n-3}} \rangle$ meets $\Gamma$, then $L$ is of LP-type.
\end{theorem}
\begin{proof}
An integer $s$ exists such that $(s,n)=1$ and
$\sigma=\hat{\sigma}^s$, i.e. the $i$-th component
(\footnote{Starting to count from zero.}) of
$\la(x_0,x_1,\ldots,x_{n-1})\ra_{\F_{q^n}}^\sigma$ is
$x_{i+s}^{q^s}$, where $i+s$ is seen modulo $n$.
By Theorem \ref{k-2}, there exist $P$ and $Q$ in $\Gamma$ such that
\[ \Gamma=\la P,P^\sigma,\ldots,P^{\sigma^{n-4}},Q \ra. \]
Denote by $R=P^{\sigma^{n-2}}$, then
\[ \Gamma=\la R^{\sigma^2},R^{\sigma^3},\ldots,R^{\sigma^{n-2}},Q \ra, \]
and $Q$ may be chosen in $\la R^\sigma, R^{\sigma^{n-1}}\ra$.
If $\dim \langle R,R^\sigma,\ldots,R^{\sigma^{n-1}}\rangle<n-1$, then, since $Q \in \la R^\sigma, R^{\sigma^{n-1}}\ra$, it follows that
\[ \Gamma \subseteq \langle R,R^\sigma,\ldots,R^{\sigma^{n-1}}\rangle, \]
i.e.\ $\Gamma$ is contained in a subspace fixed by $\sigma$ of dimension either $n-3$ or $n-2$.
In both the cases we get a contradiction because of $\mathrm{intn}_{\sigma}(\Gamma)=2$.
So, $\dim \la R,R^\sigma,\ldots,R^{\sigma^{n-1}}\ra=n-1$, and by \cite[Proposition 3.1]{BoPol}
there exists a linear collineation $\mathbf{k}$ fixing $\Sigma$ such that
$R^{\mathbf{k}}=\la(1,0,\ldots,0)\ra_{\F_{q^n}}$.
Clearly, $\Gamma^{\mathbf{k}}$ satisfies the same hypothesis as $\Gamma$, since $\mathbf{k}$ and $\sigma$ commute.
For these reasons, we may assume that $R=\la(1,0,\ldots,0)\ra_{\F_{q^n}}$.
In particular, it follows that the coordinates of $R^{\sigma^i}$ are $\mathbf{e}_{is \pmod{n}}$, where $\mathbf{e}_j$ is the vector whose $j$-th component is one and all the others are zero.
And by hypothesis we may assume that
$Q=\la\mathbf e_s-\delta\mathbf e_{s(n-1)}\ra_{\F_{q^n}}$.
Hence we can choose as $\Lambda=\la R,R^{\sigma^{n-1}} \ra$, so $\Gamma$ has equations
$x_0=0$, $x_{s(n-1)}=-\delta x_s$,
and $\Lambda$ is defined by $x_i=0$ for $i \in \{s,\ldots,s(n-2)\}$.

Therefore,
\[ L=\p_{\Gamma,\Lambda}(\Sigma)\simeq \{ \la (x,\delta x^{q^s}+x^{q^{s(n-1)}}) \ra_{\F_{q^n}} \colon x \in \F_{q^n}  \}, \]
i.e. $L$ is of LP-type.
\end{proof}

Each linear set of LP-type $L_{s,\delta}^n$ \eqref{LPform} of $\PG(1,q^n)$, with $n\geq 4$ and $(s,n)=1$,
can be realized as the projection of $\Sigma$ choosing $\Gamma$ and $\Lambda$ as follows
\[ \Gamma \colon \left\{ \begin{array}{llr} x_0=0 \\ x_{s(n-1)}=-\delta x_s \end{array} \right. \,\,\, \text{and} \,\,\, \Lambda\colon x_i=0, \quad i \in \{s,\ldots,s(n-2)\}. \]
Therefore, as a direct consequence of Theorem \ref{charact1} we provide a characterization result of linear sets of LP-type.

\begin{theorem}\label{charact1.1}
Let $\Sigma$ be a canonical subgeometry of $\PG(n-1,q^n)$, $q>2$ and $n\geq 4$.
Let $L$ be a scattered linear set in $\Lambda=\PG(1,q^n)$.
Then $L$ is a linear set of LP-type if and only if
\begin{enumerate}[(i)]
\item there exists an $(n-3)$-subspace $\Gamma$  of $\PG(n-1,q^n)$
such that $\Gamma \cap \Sigma=\Gamma\cap\Lambda=\emptyset$ and $L=\p_{\Gamma,\Lambda}(\Sigma)$;
\item there exists a generator $\sigma$ of the subgroup of $\mathrm{P}\Gamma\mathrm{L}(n,q^n)$ fixing $\Sigma$ pointwise, such that
$\mathrm{intn}_{\sigma}(\Gamma)=2$;
\item there exist a unique point $P\in\PG(n-1,q^n)$ and some point $Q$ such that
\[ \Gamma=\langle P,P^{\sigma},\ldots,P^{\sigma^{n-4}},Q \rangle; \]
\item the line $\langle P^{\sigma^{n-1}}, P^{\sigma^{n-3}}\rangle$ meets $\Gamma$.
\end{enumerate}
\end{theorem}

\subsection{Necessary conditions}

Very recently, Csajb\'ok, Marino and Polverino in \cite{CMP} have investigated the equivalence problem between $\F_q$-linear sets of rank $n$ on the projective line $\PG(1,q^n)$.
The idea is first to study the $\Gamma\mathrm{L}(2,q^n)$-orbits of the subspace $U$ defining the linear set $L_U$ and then to study the equivalence between two linear sets.
More precisely, they give the following definition of $\Gamma\mathrm{L}$-class (see \cite[Definitions 2.5]{CMP}) of an $\F_q$-linear set of a line.

\smallskip

Let $L_U$ be an $\mathbb{F}_q-$linear set of $\PG(V,\mathbb{F}_{q^n})=\PG(1,q^n)$ of rank $n$ with maximum field of linearity $\mathbb{F}_q$ (\footnote{The \emph{maximum field of linearity} of an $\F_q$-linear set $L_U$ as $\F_{q^\ell}$ if $\ell$ is the largest integer such that $\ell \mid n$ and $L_U$ is an $\F_{q^\ell}$-linear set.}).

We say that $L_U$ is of $\Gamma\mathrm{L}$-\emph{class} $s$ if $s$ is the greatest integer such that there exist $\mathbb{F}_q$-subspaces $U_1,\ldots,U_s$ of $V$ with $L_{U_i}=L_U$ for $i \in \{1,\ldots,s\}$ and there is no $f \in \Gamma \mathrm{L}(2,q^n)$ such that $U_i=U_j^f$ for each $i\neq j$, $i,j \in \{1,2,\ldots,s\}$.

\smallskip

If $L_U$ is of $\Gamma \mathrm{L}$-class one, then $L_U$ is said to be \emph{simple}, i.e. when the $\Gamma\mathrm{L}(2,q^n)$-orbit of $U$ completely determine the $\mathrm{P}\Gamma\mathrm{L}(2,q^n)$-orbit of $L_U$.
For $n\le4$, any linear set in $\PG(1,q^n)$ is simple \cite[Theorem 4.5]{CMP}.

The $\Gamma\mathrm{L}$-class of a linear set is a projective invariant (by \cite[Proposition 2.6]{CMP}) and play a crucial role in the study of linear sets up to equivalences.
Using these notions and by developing some new techniques, the authors in \cite{CsMP2018} prove that in $\PG(1,q^5)$ each $\F_q$-linear set $L_f$ of rank $5$ and with maximum field of linearity $\F_q$ is of $\Gamma\mathrm{L}$-class at most $2$, proving also that if $L_U$ is equivalent to $L_f$ then $U$ is $\Gamma\mathrm{L}$-equivalent to either $U_f$ or to $U_f^\perp=U_{\hat{f}}$, where the non-degenerate symmetric bilinear form of $\F_{q^n}$ over $\F_q$ defined by
\[ \la x,y\ra= \Tr_{q^n/q}(xy), \]
\noindent for each $x,y \in \F_{q^n}$ is taken into account. Then the \emph{adjoint} $\hat{f}$ of the linearized polynomial $\displaystyle f(x)=\sum_{i=0}^{n-1} a_ix^{q^i} \in \tilde{\mathcal{L}}_{n,q}$ with respect to the bilinear form $\la,\ra$ is
\[ \hat{f}(x)=\sum_{i=0}^{n-1} a_i^{q^{n-i}}x^{q^{n-i}}, \]
i.e.
\[ \Tr_{q^n/q}(xf(y))=\Tr_{q^n/q}(y\hat{f}(x)), \]
for each $x,y \in \F_{q^n}$.

\smallskip

For linear sets of LP-type the following is known.

\begin{theorem}\label{classLP}\cite{CsMP2018,CsMZ2018}
A maximum scattered linear set of LP-type
\[
L_{s,\delta}^n=L_f=\{\la(x,\delta x^{q^s} + x^{q^{n-s}})\ra_{\F_{q^n}}\colon x\in \F_{q^n}\}\subseteq\PG(1,q^n), \]
with $(s,n)=1$ and $f(x)=\delta x^{q^s} + x^{q^{n-s}}$,
is of $\Gamma\mathrm{L}$-class less than or equal to $2$ for $n \in \{5,6,8\}$.
Furthermore, $L_U$ is equivalent to $L$ if and only if $U$ is $\Gamma\mathrm{L}(2,q^n)$-equivalent to either $U_{f}$ or to $U_f^\perp=U_{\hat{f}}$.
\end{theorem}

Furthermore, in \cite{CMP,PhDthesis}, it has been shown that there are maximum scattered linear sets of LP-type of both $\Gamma\mathrm{L}$-classes one and two.

\medskip

For our purpose it is important to look to the $\Gamma\mathrm{L}$-class in a more geometric way.
The following result has been stated in \cite[Section 5.2]{CMP} as a consequence of \cite[Theorems 6 \& 7]{CSZ2015}.

\begin{theorem}\label{GLclassGeom}
The $\Gamma\mathrm{L}$-class of $L_U$ is the number of orbits in $\mathrm{Stab}(\Sigma)$ of  $(n-3)$-subspaces of $\PG(n-1,q^n)$ containing a $\Gamma$ disjoint from $\Sigma$ and from $\Lambda$ such that $\p_{\Gamma,\Lambda}(\Sigma)$ is equivalent to $L_U$.
\end{theorem}

As a consequence of Theorems \ref{classLP} and \ref{GLclassGeom}, we have the following characterization for linear sets of LP-type.

\begin{theorem}\label{charact3}
Let $L$ be a maximum scattered linear set in $\Lambda=\PG(1,q^n)$ with $n\le6$ or $n=8$.
Then $L$ is a linear set of LP-type if and only if for each $(n-3)$-subspace $\Gamma$ of $\PG(n-1,q^n)$ such that $L=\p_{\Gamma,\Lambda}(\Sigma)$, the following holds:
\begin{enumerate}[(i)]
\item there exists a generator $\sigma$ of the subgroup of $\mathrm{P}\Gamma\mathrm{L}(n,q^n)$ fixing $\Sigma$ pointwise, such that
$\mathrm{intn}_{\sigma}(\Gamma)=2$;
\item if $P$ is the unique point of $\PG(n-1,q^n)$ such that
\[ \Gamma=\langle P,P^{\sigma},\ldots,P^{\sigma^{n-4}},Q \rangle, \]
then the line $\langle P^{\sigma^{n-1}}, P^{\sigma^{n-3}}\rangle$ meets $\Gamma$.
\end{enumerate}
\end{theorem}
\begin{proof}
Because of Theorem \ref{classLP} and \cite[Theorem 4.5]{CMP}, if $n\le6$
or $n=8$, then the two (possibly) not $\Gamma\mathrm{L}(2,q^n)$-equivalent representation for a linear set of LP-type \eqref{LPform} are
\[ U_{\delta x^{q^s} + x^{q^{n-s}}} \,\,\,\text{and}\,\,\,U_{\delta^{q^{n-s}} x^{q^{n-s}} + x^{q^{s}}}.  \]
Therefore, by Theorem \ref{GLclassGeom} we have that all the possible vertices of the projections to obtain a linear set of LP-type satisfy the hypothesis of Theorem \ref{charact1.1} and the assertion then follows.
\end{proof}

\begin{remark}
Note that Theorem \ref{charact3} guarantees that each vertex of the projection of a linear set of LP-type satisfies conditions $(i)$ and $(ii)$, whereas Theorem \ref{charact1.1} asserts the existence of a vertex of the projection of a linear set of LP-type satisfying these conditions.
\end{remark}

As we will see in Section \ref{newconstruction}, this result may turn out to be useful to construct new examples of maximum scattered linear sets in $\PG(1,q^n)$.

\section{A purely geometric description for odd $n$}\label{Geo}

The next lemma proves that, for $n$ odd, the only scattered linear sets of LP-type are exactly those described by Lunardon and Polverino in \cite{LP2001} and Sheekey \cite{Sh}.

\begin{lemma}\label{norm}
Let $L:=\{\la(x,\delta x^{q^s} + x^{q^{n-s}})\ra_{\F_{q^n}}\colon x\in \F_{q^n}\}\subseteq \PG(1,q^n)$, with $(s,n)=1$, and let $n>3$ be odd.
Then $L$ is scattered if and only if $\N(\delta)\neq 1$.
\end{lemma}
\begin{proof}
We only have to prove that if $\N(\delta)= 1$, then $L$ cannot be scattered.
The linear set $L$ is scattered if and only if
in the following set of polynomials
\[ A=\{\alpha x+\delta x^{q^s} + x^{q^{n-s}} \colon \alpha\in \F_{q^n}\}  \]
there are no polynomials with more than $q$ roots, for otherwise there would be a point
$\langle(1,-\alpha)\rangle_{\F_{q^n}}$ of weight greater than one.
Equivalently in the following set of polynomials
\[ A'=\{f_{\alpha}(x)=\delta^{-1}x+\alpha x^{q^{s}}+ x^{q^{2s}} \colon \alpha\in \F_{q^n}\}  \]
there are no polynomials with more than $q$ roots.
For any $\xi \in \F_{q^n}^*$ with $\N(\xi)=1$, the polynomial
\begin{equation}\label{polform}
\frac{f_{\alpha}(x) \circ \xi x}{\xi^{q^{2s}}}=\delta^{-1} \xi^{1-q^{2s}}x+\alpha \xi^{q^s-q^{2s}}x^{q^s}+x^{q^{2s}}
\end{equation}
has the same number of roots of $f_{\alpha}(x)$.
Note that since $n$ is odd, for any $m\in\Fqn$ such that
$\NNN(m)=1$ there is $\xi\in\Fqn$ such that $m=\xi^{1-q^{2s}}$.
Taking into account $\NNN(\delta)=1$,
this implies that for any polynomial of the form $P(x)=\gamma x+\beta x^{q^s}+x^{q^{2s}}$,
with $\gamma,\beta\in\Fqn$ and $\N(\gamma)=1$, there are $\alpha$ and $\xi\in\Fqn$ such that
\eqref{polform} coincides with $P(x)$.
This is a contradiction, since there exist polynomials
of type $\gamma x+\beta x^{q^s}+x^{q^{2s}}$, $\NNN(\gamma)=1$, with $q^2$ roots, e.g.
\[ \frac{1}{u^{q^s}v^{q^{2s}}-u^{q^{2s}}v^{q^s}}\det \left( \begin{array}{cccccc} x & x^{q^s} & x^{q^{2s}} \\ u & u^{q^s} & u^{q^{2s}} \\ v & v^{q^s} & v^{q^{2s}} \end{array}\right) \]
where $u,v \in \F_{q^n}$ are $\F_q$-linearly independent.
\end{proof}

The previous lemma was already proved for $n=4$ in \cite{CsZ2018} and for $s=1$ in \cite{BartoliZhou}.

\begin{theorem}\label{charact2}
Let $\Gamma$ be a subspace of $\PG(n-1,q^n)$, $n$ odd, of dimension $n-3\geq 2$, and
$\Sigma\cong\PG(n-1,q)$ a canonical subgeometry of $\PG(n-1,q^n)$, such that
$\Gamma \cap \Sigma=\emptyset$.
Assume that a generator $\sigma$ exists of the subgroup of $\mathrm{P}\Gamma\mathrm{L}(n,q^n)$ fixing $\Sigma$ pointwise, such that $\mathrm{intn}_{\sigma}(\Gamma)=2$.
Then there exists a point $R\in\PG(n-1,q^n)$ such that
\[ R^{\sigma^2},R^{\sigma^3},\ldots,R^{\sigma^{n-2}}\in\Gamma. \]
Furthermore assume that $\la R^\sigma,R^{\sigma^{n-1}}\ra$ and $\Gamma$ meet in a point $Q$ and
$R^\sigma\neq Q\neq R^{\sigma^{n-1}}$.
Let $Q'$ be the point such that the pair $\{R^\sigma,R^{\sigma^{n-1}}\}$ separates $\{Q,Q'\}$
harmonically.
Such $Q'$ is defined by the property that there are two representative vectors $v_0$ and $v_1$ for
$R^\sigma$ and $R^{\sigma^{n-1}}$, respectively, such that $\la v_0+v_1\ra_{\Fqn}=Q$,
$\la v_0-v_1\ra_{\Fqn}=Q'$.
Under the assumptions above, the linear set $L=\p_{\Gamma,\Lambda}(\Sigma)$, with $\Lambda$ a line disjoint from $\Gamma$,
is a maximum scattered linear set of LP-type
if and only if
\begin{equation}\label{strana}
\Sigma\cap\la R,R^{\sigma^2},R^{\sigma^3},\ldots,R^{\sigma^{n-2}},Q'\ra=\emptyset.
\end{equation}
\end{theorem}
\begin{proof}
As in Theorem \ref{charact1} it may be assumed that the coordinates of $Q$ are
$x_s=1$, $x_{s(n-1)}=-\delta\in\F_{q^n}^*$, $x_i=0$ otherwise.
The coordinates of $Q'$ are
$x_s=1$, $x_{s(n-1)}=\delta\in\F_{q^n}^*$, $x_i=0$ otherwise.
The span $W=\la R,R^{\sigma^2},R^{\sigma^3},\ldots,R^{\sigma^{n-2}},Q'\ra$ and $R^\sigma$
are complementary subspaces of $\PG(n-1,q^n)$.
So, $W$ is a hyperplane and its equation is $-\delta x_s+x_{s(n-1)}=0$.
A point $\la(u,u^q,\ldots,u^{q(n-1)})\ra_{\F_{q^n}}$ of $\Sigma$ belongs to $W$ if and only if
$-\delta u^{q^s}+u^{q^{s(n-1)}}=0$, equivalent to
\begin{equation}\label{menodelta}
\delta=u^{q^s(q^{s(n-2)}-1)}.
\end{equation}
Since $n$ is odd, $s(n-2)$ is coprime with $n$.
This implies that an $u\in\F_{q^n}^*$ exists satisfying (\ref{menodelta}) if and only if
$\N(\delta)=1$, which is a contradiction because of Lemma \ref{norm}.
\end{proof}

\section{New constructions}\label{newconstruction}

In this section we will deal with the following family of linear sets
\[ \mathcal{L}:=\{ \langle (x,x^q-x^{q^2}+x^{q^4}+x^{q^5})\rangle_{\F_{q^6}} \colon x \in \F_{q^6}^* \} \subseteq \PG(1,q^6),\quad q \equiv 1 \pmod{4}. \]
We will show that for some choices of $q$ we may get new examples of maximum scattered linear sets.
This family of linear sets can be obtained by projecting the canonical subgeometry $\Sigma=\{ \langle(x,x^q,x^{q^2},x^{q^3},x^{q^4},x^{q^5})\rangle_{\F_{q^6}} \colon x \in \F_{q^6}^* \}$ from
\[ \Gamma \colon \left\{ \begin{array}{ll} x_0=0 \\ x_5=-x_4-x_1+x_2 \end{array} \right. \]
to
\[ \Lambda \colon \left\{ \begin{array}{llll} x_1=0\\ x_2=0\\ x_3=0\\ x_4=0. \end{array} \right. \]
Let us consider $\sigma\in \mathrm{P}\Gamma\mathrm{L}(6,q^6)$ defined as
\[\left(\la(x_0,x_1,x_2,x_2,x_4,x_5)\ra_{\F_{q^6}}\right)^\sigma=
\la(x_5^q,x_0^q,x_1^q,x_2^q,x_3^q,x_4^q)\ra_{\F_{q^6}}\]
and $\overline{\sigma}:=\sigma^5$, which are the two generators of the subgroup of $\mathrm{P}\Gamma\mathrm{L}(6,q^6)$ fixing $\Sigma$ pointwise.
Then
\[ \Gamma^{\sigma} \colon \left\{ \begin{array}{ll} x_1=0 \\ x_0=-x_5-x_2+x_3 \end{array} \right. \,\,\, \text{and} \,\,\, \Gamma^{\sigma^2} \colon \left\{ \begin{array}{ll} x_2=0 \\ x_1=-x_0-x_3+x_4. \end{array} \right. \]
Therefore,
\[ \Gamma\cap\Gamma^\sigma \colon \left\{ \begin{array}{llll} x_0=0\\ x_1=0 \\ x_4=2x_2-x_3 \\ x_5=-x_2+x_3 \end{array} \right. \,\,\, \text{and} \quad \Gamma\cap\Gamma^\sigma\cap\Gamma^{\sigma^2} =\emptyset. \]
Hence, $\dim_{\F_{q^6}} (\Gamma \cap \Gamma^\sigma)=1$ and
since $q$ is odd $\dim_{\F_{q^6}} (\Gamma \cap \Gamma^\sigma\cap \Gamma^{\sigma^2})= -1$.
Since $\Gamma \cap \Gamma^{\overline{\sigma}}=(\Gamma\cap\Gamma^\sigma)^{\sigma^5}$ and $\Gamma \cap \Gamma^{\overline{\sigma}}\cap \Gamma^{\overline{\sigma}^2} =(\Gamma\cap\Gamma^\sigma\cap\Gamma^{\sigma^2})^{\sigma^4}$, we have that
$\dim_{\F_{q^6}} (\Gamma \cap \Gamma^{\overline{\sigma}})=1$ and $\dim_{\F_{q^6}} (\Gamma \cap \Gamma^{\overline{\sigma}}\cap \Gamma^{\overline{\sigma}^2})= -1$.
Therefore,
\[ \mathrm{intn}_{\sigma}(\Gamma)= \mathrm{intn}_{\overline{\sigma}}(\Gamma)= 3. \]
This implies the non-equivalence of $\mathcal L$
with the linear set of pseudoregulus type and also it cannot be of LP-type because of Theorem \ref{charact3}.

Computational results show that for $q \equiv 1 \pmod{4}$ the linear set $\mathcal{L}$ is maximum scattered for $q\leq 29$.
We show that for $q\leq 17$ and $q \not\equiv 0 \pmod{5}$ it is also new.
For $q\equiv 0 \pmod{5}$ we will prove in Proposition \ref{trin0}
that $\mathcal{L}$ is equivalent to the linear set defined in \cite{CsMZ2018}.

\subsection{Known examples of maximum scattered linear sets in $\PG(1,q^6)$}\label{EquivIssue}

In order to decide whether the linear set $\mathcal{L}$ is new, we describe the known maximum scattered linear sets in $\PG(1,q^6)$.

We start by listing the non-equivalent (under the action of $\Gamma\mathrm{L}(2,q^6)$) maximum scattered subspaces of $\F_{q^6}^2$, i.e.\ subspaces defining maximum scattered linear sets.

\begin{example}\label{exKnownscattered}
\begin{enumerate}
\item $U^{1}:= \{(x,x^{q}) \colon x\in \F_{q^6}\}$, see \cite{BL2000,CsZ20162};
\item $U^{2}_{\delta}:= \{(x,\delta x^{q} + x^{q^{5}})\colon x\in \F_{q^6}\}$, $\N_{q^6/q}(\delta)\notin \{0,1\}$ \footnote{This condition  implies $q\neq 2$.}, see \cite{LP2001,LTZ,Sh};
\item $U^{3}_{\delta}:= \{(x,\delta x^{q}+x^{q^{4}})\colon x\in \F_{q^{6}}\}$, $\N_{q^6/q^{3}}(\delta) \notin \{0,1\}$, satisfying further conditions on $\delta$ and $q$, see \cite[Theorems 7.1 and 7.2]{CMPZ} \footnote{Also here $q>2$, otherwise it is not scattered.};
\item $U^{4}_{\delta}:=\{(x, x^q+x^{q^3}+\delta x^{q^5}) \colon x \in \F_{q^6}\}$, $q$ odd and $\delta^2+\delta=1$, see \cite{CsMZ2018,MMZ}.
\end{enumerate}
\end{example}

In order to simplify the notation, we will denote by $L^1$ and $L^{i}_{\delta}$ the $\F_q$-linear set defined by $U^{1}$ and $U^{i}_{\delta}$, respectively.
Therefore, $L_\delta^2=L_{s,\delta}^6$ as defined in \eqref{LPform}.
We will also use the following notation: $\mathcal{U}:=U_{x^q-x^{q^2}+x^{q^4}+x^{q^5}}$.

In \cite[Propositions 3.1, 4.1 \& 5.5]{CsMZ2018} the following result has been proved.

\begin{lemma}\label{equiv}
Let $L_f$ be one of the maximum scattered of $\PG(1,q^6)$ listed before.
Then a linear set $L_U$ of $\PG(1,q^6)$ is $\mathrm{P}\Gamma\mathrm{L}$-equivalent to $L_f$ if and only if $U$ is $\Gamma\mathrm{L}$-equivalent either to $U_f$ or to $U_{\hat{f}}$.
Furthermore, the linear set $L^3_{\delta}$ is simple.
\end{lemma}

The previous lemma includes results on linear sets of LP-type.

\begin{remark}\label{closedadjoint}
If $U_f$ is an $\Fq$-subspace of type 1. or 2. above, then $U_{\hat f}$ and $U_f$ are
$\Gamma\mathrm{L}$-equivalent.
By Lemma \ref{equiv}, this holds also for $\Fq$-subspaces of type 3.

\end{remark}

\subsection{The linear set $\mathcal{L}$}

Here we deal with the equivalence issue between the linear sets defined by Example \ref{exKnownscattered} and the linear set $\mathcal{L}$.
As already noted, we just have to check the equivalence with the linear sets $L^3_{\delta}$ and with $L^4_{\delta}$ defined by the subspaces 3. and 4. in Example \ref{exKnownscattered}, because of the construction of $\mathcal{L}$ and Theorems \ref{chPseudo} and \ref{charact3}.

\begin{proposition}
The linear set $\mathcal{L}$ is not $\mathrm{P}\Gamma\mathrm{L}$-equivalent to $L^3_{\delta}$.
\end{proposition}
\begin{proof}
By Lemma \ref{equiv}, we have to check whether $\mathcal{U}$ and $U^3_{\delta}$ are $\Gamma\mathrm{L}$-equivalent, with $\N_{q^6/q^{3}}(\delta) \notin \{0,1\}$.
Suppose that there exist $\sigma \in \mathrm{Aut}(\F_{q^6})$ and an invertible matrix
$\left( \begin{array}{llrr} a & b \\ c & d \end{array} \right)$ such that for each $x \in \F_{q^6}$ there exists $z \in \F_{q^6}$ satisfying  \[ \left(
\begin{array}{llrr}
a & b\\
c & d
\end{array} \right)
\left(
\begin{array}{ccc}
\hspace{1cm}x^\sigma\\
x^{\sigma q}-x^{\sigma q^2}+x^{\sigma q^4}+x^{\sigma q^5}
\end{array} \right)
=\begin{pmatrix} z\\
{\delta z^q+ z^{q^4}} \end{pmatrix}.\]
Equivalently, for each $x \in \F_{q^6}$ we have
\[cx^\sigma+d(x^{\sigma q}-x^{\sigma q^2}+x^{\sigma q^4}+x^{\sigma q^5})=\delta [a^qx^{\sigma q}+\]
\[ +b^q(x^{\sigma q^2}-x^{\sigma q^3}+x^{\sigma q^5}+x^\sigma)]+a^{q^4}x^{\sigma q^4}+b^{q^4}(x^{\sigma q^5}-x^\sigma+x^{\sigma q^2}+x^{\sigma q^3}). \]
This is a polynomial identity in $x^{\sigma}$ and hence we have the following relations:
\begin{equation}\label{binZ}
\left\{
\begin{array}{llllll}
c=\delta b^q- b^{q^4}\\
d=\delta a^q\\
-d=\delta b^q+b^{q^4}\\
0=-\delta b^q+ b^{q^4}\\
d= a^{q^4}\\
d=\delta b^q+ b^{q^4}.
\end{array}\right.
\end{equation}

From the second and the fifth equations, if $a \neq 0$ then $\delta=(a^q)^{q^3-1}$ and so $\N_{q^6/q^3}(\delta)=1$, which is not possible.
So $a=0$ and then $d=0$. Hence we have $\delta b^q+b^{q^4}=0$ and $-\delta b^q+b^{q^4}=0$, from which we get $b=0$, which is not possible.
Therefore, $\mathcal{L}$ is not equivalent to $L^3_{\delta}$.
\end{proof}

\begin{proposition}\label{trin0}
The linear set $\mathcal{L}$ is $\mathrm{P}\Gamma\mathrm{L}$-equivalent to $L^4_{\delta}$, for $q$ odd and $\delta^2+\delta=1$, if and only if
there exist $a,b,c,d \in \F_{q^6}$ such that $ad-bc \neq 0$ and either
\begin{equation}\label{trin}
\left\{
\begin{array}{llllll}
c=b^q+\delta b^{q^5}\\
d=a^q+b^{q^3}-\delta b^{q^5}\\
-d=b^q+b^{q^3}\\
0=-b^q+a^{q^3}+\delta b^{q^5}\\
d=b^{q^3}+\delta b^{q^5}\\
d=b^q-b^{q^3}+\delta a^{q^5},
\end{array}\right.
\end{equation}
or
\begin{equation}\label{trin2}
\left\{
\begin{array}{llllll}
c=\delta b^q+ b^{q^5}\\
d=\delta a^q+b^{q^3}- b^{q^5}\\
-d=\delta b^q+b^{q^3}\\
0=-\delta b^q+a^{q^3}+ b^{q^5}\\
d=b^{q^3}+ b^{q^5}\\
d=\delta b^q-b^{q^3}+ a^{q^5}.
\end{array}\right.
\end{equation}
In particular, when $q \equiv 0 \pmod{5}$ the linear set $\mathcal{L}$ is $\mathrm{P}\Gamma\mathrm{L}$-equivalent to $L^4_{2}$.
\end{proposition}
\begin{proof}
By Lemma \ref{equiv} we have to check whether $\mathcal{U}$ is equivalent either to $U^4_{\delta}$ or to $(U^4_{\delta})^\perp$.
\noindent Suppose that there exist $\sigma \in \mathrm{Aut}(\F_{q^6})$ and an invertible matrix
$\left( \begin{array}{llrr} a & b \\ c & d \end{array} \right)$ such that for each $x \in \F_{q^6}$ there exists $z \in \F_{q^6}$ satisfying  \[ \left(
\begin{array}{llrr}
a & b\\
c & d
\end{array} \right)
\left(
\begin{array}{ccr}
\hspace{1cm}x^\sigma\\
x^{\sigma q}-x^{\sigma q^2}+x^{\sigma q^4}+x^{\sigma q^5}
\end{array} \right)
=\left(
\begin{array}{ccr}
\hspace{1cm}z\\
z^q+z^{q^3}+\delta z^{q^5} \end{array}
\right).\]
Equivalently, for each $x \in \F_{q^6}$ we have
\[ cx^\sigma+d(x^{\sigma q}-x^{\sigma q^2}+x^{\sigma q^4}+x^{\sigma q^5})=a^q x^{\sigma q}+b^q(x^{\sigma q^2}-x^{\sigma q^3}+x^{\sigma q^5}+x^{\sigma})+ \]
\[ +a^{q^3} x^{\sigma q^3}+b^{q^3}(x^{\sigma q^4}-x^{\sigma q^5}+x^{\sigma q}+x^{\sigma q^2})+ \]
\[ +\delta[a^{q^5}x^{\sigma q^5}+b^{q^5}(x^{\sigma}-x^{\sigma q}+x^{\sigma q^3}+x^{\sigma q^4})]. \]
This is a polynomial identity in $x^{\sigma}$ and hence we have the Equations \eqref{trin}.

Now, suppose that there exist $\sigma \in \mathrm{Aut}(\F_{q^6})$ and an invertible matrix
$\left( \begin{array}{llrr} a & b \\ c & d \end{array} \right)$ such that for each $x \in \F_{q^6}$ there exists $z \in \F_{q^6}$ satisfying  \[ \left(
\begin{array}{llrr}
a & b\\
c & d
\end{array} \right)
\left(
\begin{array}{ccr}
\hspace{1cm}x^\sigma\\
x^{\sigma q}-x^{\sigma q^2}+x^{\sigma q^4}+x^{\sigma q^5}
\end{array} \right)
=\left(
\begin{array}{ccr}
\hspace{1cm}z\\
\delta z^q+z^{q^3}+ z^{q^5} \end{array}
\right).\]
As before, we get the Relations \eqref{trin2}.

The second part follows from the fact that for $q \equiv 0 \pmod{5}$,
$\delta=2$,  $a=-1, b=1, c=3$ and $d=3$ satisfy (\ref{trin}).
\end{proof}

Thanks to GAP computations we are able to prove that the Systems \eqref{trin} and \eqref{trin2} have no solutions in $a,b,c$ and $d$ ($ac-bd\neq0$) for $q \leq 17$ and $q \not\equiv 0 \pmod{5}$.
Therefore, we have the following result.

\begin{corollary}\label{new}
If $q\le17$, $q\equiv 1 \pmod4$, $q\neq5$, then $\mathcal{L}$ is a maximum scattered linear set of $\PG(1,q^6)$ not equivalent to any of those listed in Example \ref{exKnownscattered}.
\end{corollary}

Recall that $\mathcal{L}$ is computationally proved to be scattered for $q\le29$, $q\equiv1\pmod4$.

We conclude this section proposing the following conjecture.

\begin{conjecture}
The linear set $\mathcal{L}$ is a new maximum scattered linear set of $\PG(1,q^6)$ for each $q$
such that $q \equiv 1 \pmod{4}$ and $q\not\equiv 0 \pmod{5}$.
\end{conjecture}

\section{MRD-codes and scattered $\F_q$-subspaces}

The most natural way to look to the connection between maximum scattered linear sets and MRD-codes is through the $\F_q$-subspaces defining such linear sets, i.e. maximum scattered $\F_q$-subspaces.
We briefly recall some basic definitions and results on rank metric codes, that have been intensively investigated for their applications in cryptography, space-time coding and distributed storage and for their links with remarkable geometric and algebraic objects (see e.g. \cite{BartoliZhou,GPT,delaCruz,Lusina,NRS,Silb}).

In 1978, Delsarte \cite{Delsarte} introduced rank metric codes as follows.
The set of $m \times n$ matrices $\F_q^{m\times n}$ over $\F_q$ is a rank metric $\F_q$-space
with rank metric distance defined by
\[d(A,B) = \mathrm{rk}\,(A-B)\]
for $A,B \in \F_q^{m\times n}$.
A subset $\C \subseteq \F_q^{m\times n}$ is called a \emph{rank metric code} (or \emph{RM}-code for short).
The \emph{minimum distance} of $\C$ is
\[d = \min\{ d(A,B) \colon A,B \in \C,\,\, A\neq B \}.\]
We are interested in $\F_q$-\emph{linear} RM-codes, i.e. for which $\C$ is an $\F_q$-linear subspace of $\F_q^{m\times n}$.
We will say that such a code has parameters $(m,n,q;d)$.
In \cite{Delsarte}, Delsarte also showed that the parameters of these codes must fulfill a Singleton-like bound, i.e.
\[ |\C| \leq q^{\max\{m,n\}(\min\{m,n\}-d+1)}. \]
When the equality holds, we call $\C$ \emph{maximum rank distance} (\emph{MRD} for short) code.

From now on, we will only consider $\F_q$-linear RM-codes of $\F_q^{n\times n}$, i.e. those which can be identified with $\F_q$-subspaces of $\mathrm{End}_{\F_q}(\F_{q^n})$.
Since $\mathrm{End}_{\F_q}(\F_{q^n})$ is isomorphic to the ring of $q$-polynomials over $\F_{q^n}$  modulo $x^{q^n}-x$, denoted by $\cL_{n,q}$, with addition and composition as operations, we will consider $\cC$ as an $\F_q$-subspace of $\cL_{n,q}$.
Given two $\F_q$-linear RM-codes, $\cC_1$ and $\cC_2$, they are \emph{equivalent} if and only if there exist $\varphi_1$, $\varphi_2\in \cL_{n,q}$ permuting $\F_{q^n}$ and $\rho\in \mathrm{Aut}(\F_q)$ such that
\[ \varphi_1\circ f^\rho \circ \varphi_2 \in \cC_2 \text{ for all }f\in \cC_1,\]
where $\circ$ stands for the composition of maps and $f^\rho(x)= \sum f_i^\rho x^{q^i}$ for $f(x)=\sum f_i x^{q^i}$.
For a rank metric code $\cC$ given by a set of linearized polynomials, its \emph{left} and \emph{right idealisers} can be defined as:
\[L(\cC)= \{ \varphi \in \cL_{n,q}\colon \varphi \circ f \in \cC \text{ for all }f\in \cC \},\]
\[R(\cC)= \{ \varphi \in \cL_{n,q}\colon f \circ \varphi \in \cC \text{ for all }f\in \cC \}.\]
If $L(\C)$ has maximum cardinality $q^n$, then we may always assume (up to equivalence) that
\[L(\C)=\mathcal{F}_n=\{\tau_{\alpha}=\alpha x \colon \alpha \in \F_{q^n}\}\simeq \F_{q^n};\]
the same holds for the right idealiser, see \cite[Theorem 6.1]{CMPZ} and \cite[Theorem 2.2]{CsMPZh}.
Hence, when the left idealiser is $\mathcal{F}_n$, $\C$ results to be closed with respect to the left composition with the $\F_{q^n}$-linear maps; while if the right idealiser is $\mathcal{F}_n$, then $\C$ is closed with respect to the right composition with the $\F_{q^n}$-linear maps.
For this reason, when $L(\C)$ (resp. $R(\C)$) is equal to $\mathcal{F}_n$ we say that $\C$ is $\F_{q^n}$-\emph{linear on the left} (resp. \emph{right})\index{$\F_{q^n}$-linear on the left}\index{$\F_{q^n}$-linear on the right} (or simply $\F_{q^n}$-\emph{linear} if it is clear from the context).

The notion of Delsarte dual code can be written in terms of
$q$-polynomials as follows,
see for example \cite[Section 2]{LTZ}.
Let $b:\cL_{n,q}\times\cL_{n,q}\to\F_q$ be the bilinear form
given by
\[ b(f,g)=\mathrm{Tr}_{q^n/q}\left( \sum_{i=0}^{n-1} f_ig_i \right) \]
where $\displaystyle f(x)=\sum_{i=0}^{n-1} f_i x^{q^i}$ and $\displaystyle g(x)=\sum_{i=0}^{n-1} g_i x^{q^i}$, and $\mathrm{Tr}_{q^n/q}$ is the trace function $\F_{q^n}\to\F_q$.
The \emph{Delsarte dual code} $\C^\perp$ of a set of $q$-polynomials $\C$ is
\[\C^\perp = \{f \in \cL_{n,q} \colon b(f,g)=0, \hspace{0.1cm}\forall g \in \C\}. \]

Only few families of MRD-codes are known, due to the results in \cite{BartoliZhou,ByrneRavagnani,H-TNRR}.
In \cite{Delsarte}, Delsarte gives the first construction for linear MRD-codes (he calls such sets \emph{Singleton systems}) from the perspective of bilinear forms. Few years later, Gabidulin in \cite[Section 4]{Gabidulin} presents the same class of MRD-codes by using linearized polynomials.
Although these codes have been originally discovered by Delsarte, they are called \emph{Gabidulin codes}.
Kshevetskiy and Gabidulin in \cite{kshevetskiy_new_2005} generalize the previous construction obtaining the so-called \emph{generalized Gabidulin codes}
\[\mathcal{G}_{k,s}=\langle x,x^{q^s},\ldots,x^{q^{s(k-1)}} \rangle_{\F_{q^n}},\]
with $\gcd(s,n)=1$ and $k\leq n-1$.
The RM-code $\mathcal{G}_{k,s}$ is an $\F_{q}$-linear MRD-code with parameters $(n,n,q;n-k+1)$ and $L(\mathcal{G}_{k,s})=R(\mathcal{G}_{k,s})\simeq \F_{q^n}$, see \cite[Lemma 4.1 \& Theorem 4.5]{LN2016}.
Note that, as proved in \cite{Gabidulin,kshevetskiy_new_2005}, this family is closed with respect to the Delsarte duality, more precisely $\mathcal{G}_{k,s}^\perp$ is equivalent to $\mathcal{G}_{n-k,s}$.
This family of MRD-codes has been characterized by Horlemann-Trautmann and Marshall in \cite{H-TM} as follows.

\begin{theorem}\label{gabidulind}\cite[Theorem 4.8]{H-TM}
  An $\F_{q^n}$-linear MRD-code $\cC \subseteq \mathcal{L}_{n,q}$ having dimension $k$ (over $\F_{q^n}$) is equivalent
  to a generalized Gabidulin code $\cG_{k,s}$ if and only if there is
  an integer $s<n$ with $\gcd(s,n)=1$ and
  $\dim_{\F_{q^n}} (\cC\cap\cC^{[s]})=k-1$, where $\cC^{[s]}=\{f(x)^{q^s} \colon f \in \cC\}$.
\end{theorem}

Very recently, Neri in \cite{Neri} removed the hypothesis on $\C$ to be an MRD-code.

Sheekey in \cite{Sh} proves that with $\gcd(s,n)=1$, the set
\[ \mathcal{H}_{k,s}(\eta,h)=\{a_0x+a_1x^{q^s}+\ldots+a_{k-1}x^{q^{s(k-1)}}+a_0^{q^h}\eta x^{q^{sk}} \colon a_i \in \F_{q^n}\}, \]
with $k\leq n-1$ and $\eta \in \F_{q^n}$ such that $\N_{q^n/q}(\eta)\neq (-1)^{nk}$, is an $\F_q$-linear MRD-code of dimension $nk$ with parameters $(n,n,q;n-k+1)$.
This code is called \emph{generalized twisted Gabidulin code}.
Lunardon, Trombetti and Zhou in \cite{LTZ}, generalizing the results of \cite{Sh}, determined the automorphism group of the generalized twisted Gabidulin codes and proved that, up to equivalence, the generalized Gabidulin codes and the twisted Gabidulin codes are both proper subsets of this class.
Clearly, for $\eta=0$ we have exactly the generalized Gabidulin code $\mathcal{G}_{k,s}$.
Also, the authors in \cite[Corollary 5.2]{LTZ} determined the left and right idealisers: if $\eta \neq 0$, then
\begin{equation}\label{leftrightidealH}
L(\mathcal{H}_{k,s}(\eta,h))\simeq\F_{q^{\gcd(n,h)}} \,\, \text{and} \,\, R(\mathcal{H}_{k,s}(\eta,h))\simeq\F_{q^{\gcd(n,sk-h)}}.
\end{equation}
The class of generalized twisted Gabidulin codes is closed with respect to the Delsarte duality, more precisely $\mathcal{H}_{k,s}(\eta,h)^\perp$ is equivalent to $\mathcal{H}_{n-k,s}(-\eta,n-h)$,
\cite[Theorem 6]{Sh} and \cite[Propositions 4.2]{LTZ}.
We are interested in the case when $h=0$, i.e.
\[  \mathcal{H}_{k,s}(\eta):=\mathcal{H}_{k,s}(\eta,0)=\langle x+\eta x^{q^{sk}}, x^{q^s},\ldots, x^{q^{s(k-1)}}\rangle_{\F_{q^n}}, \]
which is $\F_{q^n}$-linear (more precisely it is an $\F_{q}$-linear MRD-code $\F_{q^n}$-linear on the left).
This family has been characterized in \cite{GiuZ}.

\begin{theorem}\cite[Theorem 3.9]{GiuZ}\label{thm:charcGTG}
  Let $\cC \subseteq \mathcal{L}_{n,q}$ be an $\F_{q^n}$-linear MRD-code having dimension $k>2$.
  Then, the code $\cC$ is equivalent to a generalized twisted Gabidulin code if and only if there exists an integer $s$ such that $\gcd(s,n)=1$ and such that the following two conditions hold
\begin{enumerate}
  \item $\dim (\cC \cap \cC^{[s]})=k-2$ and $\dim(\cC\cap \cC^{[s]} \cap \cC^{[{2s}]})=k-3$, i.e.
        there exist $p(x),q(x) \in \cC$ such that
        \[ \cC= \langle p(x)^{q^s}, p(x)^{q^{2s}}, \ldots, p(x)^{q^{s(k-1)}} \rangle_{\F_{q^n}} \oplus \langle q(x) \rangle_{\F_{q^n}}; \]
  \item $p(x)$ is invertible and there exists $\eta \in \F_{q^n}^*$ such that $p(x)+\eta p(x)^{q^{sk}} \in \cC$.
\end{enumerate}
\end{theorem}

Apart from the two infinite families of $\F_{q^n}$-linear MRD-codes (i.e. $\cG_{k,s}$ and $\cH_{k,s}(\eta)$), there are few other examples known for $n \in \{6,7,8\}$, which arise from the connection with scattered linear sets we are going to explain.

In \cite[Section 5]{Sh} Sheekey showed that scattered $\F_q$-subspaces of $\F_{q^n}\times\F_{q^n}$ of dimension $n$ yield $\F_q$-linear MRD-codes with parameters $(n,n,q;n-1)$ with left idealiser isomorphic to $\F_{q^n}$; see \cite{CSMPZ2016,CsMPZ2019,ShVdV} for further details on such kind of connections.
Let us recall the construction from \cite{Sh}. Let $U_f:=\{(x,f(x))\colon x\in \F_{q^n}\}$ be a scattered $\F_q$-subspace of $\F_{q^n}\times\F_{q^n}$.
The set
\[
\C_f:=\langle x,f(x)\rangle_{\F_{q^n}}
\]
corresponds to a set of $n\times n$ matrices over $\F_q$ forming an $\F_q$-linear MRD-code with parameters $(n,n,q;n-1)$. Also, since $\C_f$ is an $\F_{q^n}$-subspace of $\cL_{n,q}$, its left idealiser $L(\C_f)$ is isomorphic to $\F_{q^n}$.
For further details see \cite[Section 6]{CMPZ}.
Furthermore, let $\C_f$ and $\C_h$ be two MRD-codes arising from maximum scattered subspaces $U_f$ and $U_h$ of $\F_{q^n}\times \F_{q^n}$.
In \cite[Theorem 8]{Sh} the author showed that there exist invertible matrices $A$, $B$ and $\sigma \in \mathrm{Aut}(\F_{q})$ such that $A \C_f^\sigma B=\C_h$ if and only if $U_f$ and $U_h$ are $\Gamma\mathrm{L}(2,q^n)$-equivalent, i.e. he proved that the equivalence of the rank metric codes coincides with the $\Gamma\mathrm{L}$-equivalence of the corresponding subspaces.

As a consequence we get the following result.

\begin{theorem}
\label{thm:newMRD}
If $q\leq 17$, $q \equiv 1 \pmod{4}$ and $q\neq 5$, then the RM-code $\C=\langle x, x^q-x^{q^2}+x^{q^4}+x^{q^5} \rangle_{\F_{q^6}}$ is an $\F_q$-linear MRD-code with parameters $(6,6,q;5)$ and left idealiser isomorphic to $\F_{q^6}$,
and is not equivalent to any previously known MRD-code.
\end{theorem}
\begin{proof}
From \cite[Section 6]{CMPZ}, the previously known $\F_q$-linear MRD-codes with parameters $(6,6,q;5)$ and with left idealiser isomorphic to $\F_{q^6}$ arise, up to equivalence, from one of the
maximum scattered subspaces of $\F_{q^{6}}\times\F_{q^{6}}$ described in Section \ref{EquivIssue}. From Corollary \ref{new} the result then follows.
\end{proof}

\subsection{Scattered linear sets and MRD-codes}

Lunardon in \cite[Section 3]{Lunardon2017} (see also \cite[Theorem 3.4]{ShVdV} and \cite[Section 4.1]{CsMPZ2019}) proved that if $U_f=\{(x,f(x)) \colon x \in \F_{q^n}\}$, with $f(x)=a_0x+a_1x^q+\ldots+a_{n-1}x^{q^{n-1}}$, is a scattered\footnote{The statement is more general, we have adapted it to our case.} $\F_q$-subspace of $\F_{q^n}\times\F_{q^n}$, then it can be obtained as a special quotient.
By \cite[Section 5]{Sh}, it follows that
\[ \C_f=\langle x,f(x) \rangle_{\F_{q^n}}, \]
is an MRD-code.
We may assume that the coefficient of $x$ in $f(x)$ is zero and $f(x)=x^{q^k}+\sum_{j\neq k} b_j x^{q^j}$.
Denoting with $\{i_1,\ldots,i_{n-2}\}=\{1,\ldots,k-1,k+1,\ldots,n-1\}$ and
\[ h_{i_j}(x)=x^{q^{i_j}}-b_{i_j}x^{q^k}, \,\,\, j=1,\ldots,n-2, \]
it is straightforward to see that
\[ \C^\perp=\langle h_{i_1}(x),\ldots,h_{i_{n-2}}(x) \rangle_{\F_{q^n}}. \]

We can embed $ U_{f} $ in $\F_{q^n}^n$ in such a way that the vector $(x,f(x))$ corresponds to the vector $(a_0,\ldots,a_{n-1})\in \F_{q^n}^n$ with $a_i=0$ if $i\neq 0,k$, $a_0=x$ and $a_k=f(x)$.
Note that $W=\langle U_{f} \rangle_{\F_{q^n}}$ corresponds to the $2$-dimensional subspace with equations $x_j=0$ where $j\neq 0,k$.

Let $V$ be the $\F_{q^n}$-subspace of $\F_{q^n}^n$ of dimension $n-2$ represented by the equations
\[ V \colon \left\{\begin{array}{ll} x_{0}=0 \\ x_k=-\sum_{j\neq 0,k} b_j x_j \end{array}\right.,  \]
and let $S=\{(x,x^q,\ldots,x^{q^{n-1}}) \colon x \in \F_{q^n}\}$.
Note that
\begin{equation}\label{eq:vertMRD}
V=c_{\mathcal{N}}(\C^\perp),
\end{equation}
where $c_{\mathcal{N}}(\alpha_0x+\ldots+\alpha_{n-1}x^{q^{n-1}})=(\alpha_0,\ldots,\alpha_{n-1})$.
It can be seen that $V \cap S= \{{\bf 0}\}$ and
\begin{equation}\label{eq:MRDvertex}
U_{f}= \langle V,S \rangle_{\F_q} \cap W.
\end{equation}

This link suggests a new proof of the equivalence between the assertions 1. and 2. of Theorem \ref{chPseudo}.
In the following we will assume that $L=L_f$ is a scattered linear set of $\PG(1,q^n)$ with rank $n$.

\begin{proof}(Theorem \ref{chPseudo})
Assume that $L_f$ is of pseudoregulus type, then by \cite{CSZ2015} we have that if $L_U=L_f$ then $U$ is $\Gamma\mathrm{L}(2,q^n)$-equivalent to 
\[ U_s=\{(x,x^{q^s}) \colon x \in \F_{q^n}\}.\,\,\,\text{with}\,\,\gcd(s,n)=1 \, \text{and} \, s<n/2. \]
Therefore if $U=U_s$, then $U_s= \langle V,S \rangle_{\F_q} \cap W$, with
\[ V\colon \left\{ \begin{array}{ll} x_0=0\\ x_s=0 \end{array}\right. \,\,\,\text{and}\,\,\, W\colon x_i=0\,\text{for}\,i\neq 0,s, \]
i.e. $L_f=p_{\Gamma,\Lambda}(\Sigma)$ with $\Gamma=\PG(V,\F_{q^n})=\PG(n-3,q^n)$, $\Sigma=\PG(S,\F_q)=\PG(n-1,q)$ and $\Lambda=\PG(W,\F_{q^n})=\PG(1,q^n)$.
Denote by $\sigma$ the collineation of $\PG(n-1,q^n)$ defined by
$\la(x_0,\ldots,x_{n-1})\ra_{\F_{q^n}}^{\sigma}=\la(x_{n-1}^{q},x_0^{q},\ldots,x_{n-2}^{q})\ra_{\F_{q^n}}$, which fixes precisely the points of $\Sigma$.
Therefore, we have that $\dim(\Gamma\cap\Gamma^{\sigma^s})=n-4$ and clearly $\sigma^s$ is a generator of the subgroup of $\mathrm{P}\Gamma\mathrm{L}(n,q^n)$ fixing $\Sigma$ pointwise.

Conversely, let $L=p_{\Gamma,\Lambda}(\Sigma)$ with $\dim(\Gamma\cap\Gamma^{\sigma^s})=n-4$, $\gcd(s,n)=1$, $\Gamma=\PG(V,\F_{q^n})=\PG(n-3,q^n)$, $\Sigma=\PG(S,\F_q)=\PG(n-1,q)$.
Note that $V=c_{\mathcal{N}}(\mathcal{C})$ with
\[ \cC=\langle g_1(x),\ldots,g_{n-2}(x) \rangle_{\F_{q^n}}, \]
for some linearized polynomials $g_1(x),\ldots,g_{n-2}(x)$.
It follows that
\[ V \colon \left\{ \begin{array}{ll} a_0x_0+\ldots+a_{n-1}x_{n-1}=0 \\ a_0'x_0+\ldots+a_{n-1}'x_{n-1}=0 \end{array} \right., \]
where $\C^\perp=\langle f_1(x), f_2(x)\rangle_{\F_{q^n}}$ and 
\[ f_1(x)=a_0x+\ldots+a_{n-1}x^{q^{n-1}}, \]
\[ f_2(x)=a_0'x+\ldots+a_{n-1}'x^{q^{n-1}}. \]
We may assume that $a_j=a_k'=1$ and $a_k=a_j'=0$ for some $j\neq k$, choose $W$ as the $\F_{q^n}$-subspace having equations $x_i=0$ for $i\neq j,k$.
Therefore, we have
\[ (V+S)\cap W\simeq U:=\{(f_1(x), f_2(x)) \colon x \in \F_{q^n}\}. \]
So, $L=L_U$ and $U$ results to be a scattered $\F_q$-subspace of $\F_{q^n}\times\F_{q^n}$, i.e. by \cite[Section 5]{Sh} $\cC^\perp$ is an MRD-code.
It follows that $\cC$ is an MRD-code with $\dim_{\F_{q^n}} \cC=n-2$ and $\dim_{\F_{q^n}} (\cC\cap\cC^{[s]})=n-3$.
By Theorem \ref{gabidulind}, $\C$ is equivalent to $\cG_{n-2,s}$. It follows that $U$ is $\Gamma\mathrm{L}(2,q^n)$-equivalent to $U_s$ and hence $L$ is of pseudoregulus type.
\end{proof}

In \cite{Neri}, Neri gives a characterization of generalized Gabidulin codes using the standard form of their generator matrix.
This suggests a further different approach to the characterization of linear sets of pseudoregulus type.

For linear sets of LP-type, as done for the pseudoregulus case, it follows that one of the possible $\F_q$-subspaces representing a linear set of LP-type can be obtained as in \eqref{eq:MRDvertex}, choosing $V$ in such a way that $V=c_\mathcal{N}(\mathcal{H}_{n-2,s}(\eta))$.
Since a characterization of generalized twisted Gabidulin codes is known, see Theorem \ref{thm:charcGTG} with $k=n-2$, it follows that a scattered linear set $L$ is of LP-type if and only if there exists an $\F_q$-subspace $U$ of $\F_{q^n}\times\F_{q^n}$ such that $L_U=L$, where $U$ is as in \eqref{eq:MRDvertex} and the rank-metric code associated to $V$ satisfies the hypothesis of Theorem \ref{thm:charcGTG} with $k=n-2$.
In contrast to the above characterization, those presented in the previous sections are purely geometric and take into account the problem of the possible $\F_q$-subspaces representing a linear set of LP-type.

Corrado Zanella and Ferdinando Zullo\\
Dipartimento di Tecnica e Gestione dei Sistemi Industriali\\
Universit\`a degli Studi di Padova\\
Stradella S. Nicola, 3\\
36100 Vicenza VI\\
Italy\\
\emph{\{corrado.zanella,ferdinando.zullo\}@unipd.it}

\noindent Ferdinando Zullo\\
Dipartimento di Matematica e Fisica,\\
Universit\`a degli Studi della Campania ``Luigi Vanvitelli'',\\
I--\,81100 Caserta, Italy\\
{{\em ferdinando.zullo@unicampania.it}}

\end{document}